\documentclass[12pt]{amsart}
\usepackage{amsmath}
\numberwithin{equation}{section}
\usepackage{amsthm}
\usepackage{amssymb}
\usepackage{amsfonts}
\usepackage{pinlabel}
\usepackage{graphicx}
\usepackage{relsize}
\newcommand{\nc}[2]{ \newcommand{#1}{#2} }

\nc{\avint}{ {- \hspace{-3.5mm} \int} }  

\nc{\R}{\mathrm{I \! R}}  
\nc{\N}{\mathrm{ I \! N}}  
\newcommand{\closure}[1]{ \stackrel{\rule{.1 in}{.01 in}}{#1} }
\newcommand{\ohclosure}[1]{ \stackrel{\rule{.15 in}{.01 in}}{#1} }
\newcommand{\dclosure}[1]{ \stackrel{\rule{.2 in}{.01 in}}{#1} }

\newcommand{\qclosure}[1]{ \stackrel{\rule{.4 in}{.01 in}}{#1} }
\newcommand{\qhclosure}[1]{ \stackrel{\rule{.45 in}{.01 in}}{#1} }
\newcommand{\pclosure}[1]{ \stackrel{\rule{.5 in}{.01 in}}{#1} }

\DeclareRobustCommand{\rchi}{{\mathpalette\irchi\relax}}
\newcommand{\irchi}[2]{\raisebox{\depth}{$#1\chi$}}
\newcommand{\bigrchi}{\mathlarger{\mathlarger{\rchi}}}

\newcommand{\refeqn}[1]{ (\!\!~\ref{eq:#1}) } 
\newcommand{\refthm}[1]{ \!\!~\ref{#1} }    

\nc{\Holder}{H\"{o}lder\ }

\nc{\ith}{ \ensuremath{\text{i}^{\text{th}}} }
\nc{\jth}{ \ensuremath{\text{j}^{\text{th}}} }
\nc{\kth}{ \ensuremath{\text{k}^{\text{th}}} }
\nc{\dst}{ \ensuremath{\text{1}^{\text{st}}_{\delta}} }
\nc{\dnd}{ \ensuremath{\text{2}^{\text{nd}}_{\delta}} }
\nc{\ost}{ \ensuremath{\text{1}^{\text{st}}} }
\nc{\tnd}{ \ensuremath{\text{2}^{\text{nd}}} }
\nc{\curl}{ \nabla \times }
\nc{\Div}{ \nabla \cdot }
\nc{\DC}{K}

\newcommand{\BVPc}[4]{  
  \begin{equation}
        \begin{array}{rl}
           #1 & \ \text{in}
               \ \ #4 \vspace{.05in} \\
           #2 & \ \text{on} \ \ \partial #4 \;,
        \end{array}
  \label{eq:#3}
  \end{equation}    }

\newcommand{\BVPnA}[6]{  
  \begin{equation}
        \begin{array}{rll}
           #1 \!\!\!& = #2 & \ \text{in}
               \ \ #6 \vspace{.05in} \\
           #3 \!\!\!& = #4 & \ \text{on} \ \ \partial #6 \;
        \end{array}
  \label{eq:#5}
  \end{equation}    }

\nc{\Ppl}{ \mathcal{M}^{+} }  \nc{\Pmn}{ \mathcal{M}^{-} }

\nc{\smiley}{ $\stackrel{\because}{\smile} \;$ }

\newcommand{\dist}{\mathrm{dist}}

\renewcommand{\eqref}[1]{(\ref{#1})}
\newcommand{\eat}[1]{}

\newcommand{\bsmall}{\begin{array}[c]{c}}
\newcommand{\esmall}{\end{array}}

\theoremstyle{plain}

\theoremstyle{definition}

\theoremstyle{plain}
\newtheorem{theorem}{Theorem}[section]

\newtheorem{lemma}[theorem]{Lemma}

\theoremstyle{definition}

\newtheorem{remark}[theorem]{Remark}

\def\qed{\hfill\rule{1ex}{1ex}\\}

\numberwithin{equation}{section}

\def\>{>_{\sigma}}

\title[Motion of Singular Points]{Nondegenerate Motion of Singular Points in Obstacle Problems with Varying Data}
\author[Armstrong]{Niles Armstrong}
\address{Kansas State University, Mathematics Department, 
138 Cardwell Hall, 
Manhattan, KS 66506}
 \email{niles@math.ksu.edu}
\author[Blank]{Ivan Blank}
\address{Kansas State University, Mathematics Department, 
138 Cardwell Hall, 
Manhattan, KS 66506}
 \email{blanki@math.ksu.edu}

\begin{document}
\baselineskip 18pt

\begin{abstract}
  Recent work by Serfaty and Serra give a formula for the velocity of the free
boundary of the obstacle problem at regular points (\cite{SS}), and much older work by King, Lacey, and
V\'{a}zquez gives an example of a singular free boundary point (in the Hele-Shaw flow) that remains stationary for
a positive amount of time (\cite{KLV}).  
The authors show how singular free boundaries in the obstacle problem in some settings move immediately
in response to varying data.
Three applications of this result are given, and in particular, the authors
show a uniqueness result: For sufficiently smooth elliptic divergence form operators on domains in $\R^n$ and for the
Laplace-Beltrami operator on a smooth manifold, the boundaries of distinct mean value sets (of the type found in
\cite{BH1} and \cite{BBL}) which are centered at the same point do not intersect.

\end{abstract}
\maketitle

\setcounter{section}{0}

\section{Introduction}   \label{Intro}

In his Fermi Lectures for the obstacle problem, Caffarelli stated the following theorem:
\begin{theorem} [Mean Value Theorem for Divergence Form Elliptic PDE]  \label{BCHMVT}
Let $L$ be any divergence form elliptic operator of the form $\partial_i (a^{ij}(x) \partial_j)$ with ellipticity
$[\lambda, \; \Lambda].$  For any $x_0 \in \Omega,$ there exists an increasing family $D_r(x_0)$ which
satisfies the following:
\begin{enumerate}
  \item $B_{cr}(x_0) \subset D_r(x_0) \subset B_{Cr}(x_0),$ with $c, C$ depending only on $n, \lambda,$ and $\Lambda.$
  \item For any $v$ satisfying $Lv \geq 0$ and $r<s$, we have
       \begin{equation}
              v(x_0) \leq \frac{1}{|D_r(x_0)|}\int_{D_r(x_0)} v(x) \; dx \leq  \frac{1}{|D_s(x_0)|}\int_{D_s(x_0)} v(x) \; dx .
       \label{eq:MVTres}
       \end{equation}
\end{enumerate}
Finally, for any $R > 0$ the set $D_R(x_0)$ is the noncontact set of the following obstacle problem:  \\
$u \leq G(\cdot, x_0)$ (with $G(\cdot, x_0)$ denoting the Green's function) such that
\BVPnA{L(u)}{- \bigrchi_{\{u < G\}} R^{-n}}{u}{G(\cdot,x_0)}{KeyObProb}{B_M(x_0)}
where $B_M(x_0) \subset \R^n$ and $M >0$ is sufficiently large.
\end{theorem}
\noindent
This theorem is a clear analogue of the classical mean value theorem for balls for the Laplacian, but here the role of
the balls is replaced with the sets, $D_r(x_0) = \{ u_r(x) < G(x, x_0) \}$ where $u_r$ is the solution to
Equation\refeqn{KeyObProb}with $R = r.$  Of course, this theorem immediately leads to questions about exactly
what can be said about these $D_r(x_0)$'s.  In fact, it is rather trivial to prove the converse of this theorem:
\begin{theorem}[Converse MVT for Divergence Form Elliptic PDE]  \label{ConvMVT}
If $\{D_r(x_0)\}_{ \{r > 0, x_0 \in \R^n \} }$ is a complete collection of sets obtained for a specific operator, $L,$
of the form given in Theorem\refthm{BCHMVT}\!\!, and $v$ is a function such that
       \begin{equation}
              v(x_0) \equiv \frac{1}{|D_r(x_0)|}\int_{D_r(x_0)} v(x) \; dx
       \label{eq:MVTres2}
       \end{equation}
whenever $D_r(x_0) \subset \Omega,$ then $Lv = 0$ in $\Omega.$
\end{theorem}
\noindent
The upshot is that all of the information contained in the operator must be contained within the collection of
mean value sets and vice versa.

Theorem\refthm{BCHMVT}was proven by Blank and Hao in detail in \cite{BH1}, and an analogous result was shown in \cite{BBL}
for the Laplace-Beltrami operator on Riemannian manifolds.  Since then, the authors have been studying
properties of these mean value sets.  Indeed, the first author showed that even for some rather simple operators, $L,$
that can be chosen to be arbitrarily close to the Laplacian these sets are not necessarily convex
(see \cite{Ar}), and besides generalizing the result to the Laplace-Beltrami operator, the second author has 
shown some regularity properties of these sets (see \cite{BBL} and \cite{AB}).  One important question that has
remained is the following one:  If $y \in \partial D_r(x_0),$ then is it possible for there to exist an $\tilde{r} \ne r$ such
that $y \in \partial D_{\tilde{r}}(x_0)$?  One can restate this question as whether $r < s$ merely implies
$D_r(x_0) \subset D_s(x_0),$ or does it imply the stronger result: $\qhclosure{D_r(x_0)} \; \subset D_s(x_0)$?  In this
paper we show that in the case of the Laplace-Beltrami operator on Riemannian manifolds, and in the case where the
operator $L$ has coefficients in $C^{1,1},$ the answer is yes.  (When the coefficients are not $C^{1,1}$ the question
is still open.)

In fact, there is \textit{almost} a proof based on the Hopf Lemma that has already been pointed out in \cite{BBL}.
Along these lines, if we assume that $y_0 \in \partial D_r(x_0) \cap \partial D_s(x_0),$ and we assume that
$\partial D_r(x_0)$ is regular at $y_0,$ then by invoking Caffarelli's famous free boundary regularity theorem for
the obstacle problem (see \cite{C1} and/or \cite{C2}), then we are guaranteed that there will exist a ball
$B_{\rho}(z_0)$ satisfying:
\begin{enumerate}
    \item $B_{\rho}(z_0) \subset D_r(x_0),$ and
    \item $\partial B_{\rho}(z_0) \cap \partial D_r(x_0) \cap \partial D_s(x_0) = y_0.$
\end{enumerate}
Now if we let $u_r$ and $u_s$ be the solutions to the problem in Equation\refeqn{KeyObProb}when $R = r$ and 
$R = s$ respectively, then it follows that $v(x) := u_r(x) - u_s(x)$ will satisfy the following:
\begin{enumerate}
     \item $v \geq 0$ in $B_{\rho}(z_0),$
     \item $Lv = s^{-n} - r^{-n} < 0$ in $B_{\rho}(z_0),$ and
     \item $v(y_0) = 0.$
\end{enumerate}
In this situation we can apply the Hopf Lemma to guarantee that $\nabla v(y_0) \ne 0.$  On the other hand
$$\nabla v(y_0) = \nabla u_r(y_0) - \nabla u_s(y_0) = \nabla G(x_0, y_0) - \nabla G(x_0, y_0) = 0$$
which gives us a contradiction.

Of course the bad news in the ``proof'' above is that we assumed that $y_0$ was a regular point of the
free boundary.  Now in the most typical pictures of free boundaries with singular points, it should be even easier
to touch the boundary of $D_r(x_0)$ with a ball, in spite of these examples, Schaeffer gave other examples of
contact sets in the obstacle problem with cantor-like structures (see \cite{S}) and the recent work of Figalli and
Serra that yields some nice regularity results for the singular set seems to require that the operator be the
Laplacian (see \cite{FS}).

One can also ask if it is possible to repair the proof above so that it continues to hold even at the singular points,
and indeed, that was our first attempt at solving this problem.  In joint work by Alvarado, Brigham, Maz'ya,
Mitrea, and Ziad\'{e}, a sharp form of Hopf's Lemma is shown which does not require touching with a ball; one only
needs to touch with a ``pseudoball'' (see \cite[Theorem 4.4]{ABMMZ}).  Furthermore in Caffarelli's original 1977
Acta paper, he shows ``almost convexity'' conditions which guarantee the existence of a half ball contained in
the noncontact set (see \cite[Corollary 1 and Corollary 2]{C1}).  Unfortunately, the union of the half balls described by
Caffarelli does not contain a pseudoball of the type described in \cite{ABMMZ}, so it appears that this route will not lead
to a proof.

Now of course, one thing that really \textit{is} shown by the argument above is that if there is a situation
with $r \ne s$ and where $\partial D_r(x_0) \cap \partial D_s(x_0)$ is nonempty, then it can only happen at
singular points.  In this respect, and in viewing the flow of $\partial D_t(x_0)$ as we vary $t,$ this situation
should be compared to the results of King, Lacey, and V\'{a}zquez for the Hele-Shaw problem (see \cite{KLV}).
They show that if there is a corner built into the initial data, and if the angle formed satisfies certain inequalities,
then that corner will remain motionless for a while at the beginning of the evolution.  It is also worth observing that
Serfaty and Serra have shown a normal velocity formula for the free boundary in the obstacle problem at regular
points when varying the data, but their result (a \textit{normal} velocity formula) obviously cannot be applied at
a point of the free boundary that does not have a normal vector (see \cite{SS}).

Another attack which could lead to a full proof via the Hopf Lemma would be to expand the work of Figalli and
Serra (\cite{FS}) to include more than the Laplacian, so that the better regularity allowed us to touch the singular
set with an interior ball.  Although even the Figalli/Serra results allow for some lower-dimensional ``anomalous''
points that would need to be handled in order to get a touching ball, so generalization \textit{and} improvement
would be needed for that route, and it is quite likely impossible.  In any case, an examination of the methods
employed in their work reveals arguments that seem to be particular to the Laplacian, and so perturbation
arguments seem like a better attack as opposed to trying to do their work from scratch in a more general setting.
Even though we have not successfully expanded that work, that perturbation approach is related to our third
application of the main idea in this work.   The difficulty there is related to the instability of singular free
boundaries.  Certainly it is a trivial matter to make a singular free boundary that disappears under an appropriate
perturbation.  For example, $u(x) = x^2$ satisfies $\Delta u = \bigrchi_{\{u > 0 \}} f,$
with $f(x) \equiv 2,$ but if you raise the boundary data and/or reduce $f$ anywhere and solve the new
problem, then the free boundary will disappear.  That observation led us to the question of whether or not
we could find a way to make specific perturbations which always led to singular points.  In the third application,
although we do not get results which are precise enough to allow us to generalize \cite{FS}, we do successfully
find a way to approximate singular free boundaries with other singular free boundaries of solutions to obstacle
problems with operators with constant coefficients and which have similar boundary data, and this approximation
may be of independent interest.

In all three applications what has worked is the following idea:  We use the derivative of the solution to the obstacle
problem as a barrier, and using that function we can come to a contradiction where a related function that we can show is
nonnegative must also become negative due to standard regularity and nondegeneracy estimates that we have for
the obstacle problem if the two distinct free boundaries share a common boundary point.  We present three applications
of this idea in this paper, and the first two are simple enough to state here:

\begin{theorem}[Compact Containment of Mean Value Sets, Part I]   \label{CCMVT1}
Under the assumptions of Theorem\refthm{BCHMVT}along with the
assumption that the $a^{ij}$ belong to $C^{1,1}$ the family $\{D_r(x_0)\}$ is always \textit{strictly}
increasing in the sense that $r < s$ implies $\qhclosure{D_r(x_0)} \; \subset D_s(x_0).$
\end{theorem}
\noindent
Now within \cite{BBL} part of the main result states the following:
\begin{theorem}[Mean Value Theorem on Riemannian manifolds]  \label{RiemMVT}
Given a point $x_0$ in a complete Riemannian manifold
$\mathcal{M}$ (possibly with boundary), there exists a maximal number $r_0 > 0$ (which is
finite if $\mathcal{M}$ is compact) and a family of open sets $\{ D_r(x_0) \}$ for $0 < r < r_0,$
such that
\begin{itemize}
   \item[(A)] $0 < r < s < r_0$ implies, $D_r(x_0) \subset D_s(x_0),$ and 
   \item[(B)] $\lim_{r \downarrow 0} \max\dist_{x_0}(\partial D_r(x_0)) = 0,$ and
   \item[(C)] if $u$ is a subsolution of the Laplace-Beltrami equation, then
$$u(x_0) = \lim_{r \downarrow 0} \frac{1}{|D_r(x_0)|} \int_{D_r(x_0)} u(x) \; dx \;,$$
and $0 < r < s < r_0$ implies
$$\frac{1}{|D_r(x_0)|} \int_{D_r(x_0)} u(x) \; dx \leq \frac{1}{|D_s(x_0)|} \int_{D_s(x_0)} u(x) \; dx \;.$$
\end{itemize}
Finally, if $r < r_0,$ then the set $D_r(x_0)$ is uniquely determined as the noncontact set of
any one of a family of obstacle problems.  In fact, as long as the set $S \subset \mathcal{M}$ is ``big enough,''
then $D_R(x_0)$ is the noncontact set of the following obstacle problem:
\BVPc{\Delta_g u = - \bigrchi_{\{u < G\}} R^{-n}}{u = G(\cdot,x_0)}{KeyObProbMani}{S}
where $G$ is the Green's function for the Laplace-Beltrami operator on $S.$
\end{theorem}
\begin{remark}[A Few More Details]   \label{AFD}
A simple criteria to determine if $S$ is ``big enough'' in the theorem above is to see if
$\{u < G \} \subset \subset S.$  Assuming that it is, then the resulting noncontact set is indeed a mean value
set.  A simple exercise shows that if $S$ is big enough and $S \subset S^{\prime} \subset \mathcal{M}$ then
solving Equation\refeqn{KeyObProbMani}with $S$ replaced by $S^{\prime}$ leads to the exact same noncontact
set.  (The new Green's function and the new solution change by the exact same harmonic function.)  In fact, it
will frequently be more convenient to work with the height function, and so we define
$w_r(x) := G(x_0,x) - u_r(x)$ which obeys either:
\begin{equation}
    L w_r = \bigrchi_{\{w > 0 \}} r^{-n} - \delta_{x_0}
\label{eq:wLeqn}
\end{equation}
or
\begin{equation}
    \Delta w_r = \bigrchi_{\{w > 0 \}} r^{-n} - \delta_{x_0}
\label{eq:wLBeqn}
\end{equation}
according to which case we are currently studying.  (We use $\delta_{x_0}$ to denote the usual delta
function at $x_0.$)
\end{remark}

In the current work we can show the following:
\begin{theorem}[Compact Containment of Mean Value Sets, Part II]   \label{CCMVT2}
Under the assumptions of Theorem\refthm{RiemMVT}the family $\{D_{r}(x_0)\}$ is always \textit{strictly}
increasing in the sense that 
\begin{itemize}
   \item[(A$^{\prime}$)] $0 < r < s < r_0$ implies $\qhclosure{D_{r}(x_0)} \; \subset D_{r}(x_0).$
\end{itemize}
\end{theorem}
\noindent
In both cases, we can say that given a point $y_0$ contained in a mean value set around $x_0,$ there is a unique $r$
such that $y_0 \in \partial D_r(x_0).$


\subsection*{Acknowledgements} The authors thank Michael Hill, Rustam Sadykov, and especially
Luis Silvestre for helpful discussions and valuable input.  

\section{Notation, Conventions, and a Preliminary Lemma}  \label{NotConv}

We will use the following basic notation and assumptions throughout the paper:
$$
\begin{array}{lll}
\mathcal{M} & \ & \text{a smooth connected Riemannian n-manifold} \\
g & \ & \text{the metric for our ambient manifold} \ \mathcal{M} \\
\bigrchi_D & \ & \text{the characteristic function of the set} \ D \\
\closure{D} & \ & \text{the closure of the set} \ D \\
\text{int}(D) & \ & \text{the interior of the set} \ D \\
\partial D & \ & \text{the boundary of the set} \ D \\
\Omega(w) & \ & \{x:  w(x) > 0 \} \\
\Lambda(w) & \ & \{x: w(x) = 0 \} \\
\text{FB}(w) & \ & \partial \Omega(w) \cap \partial \Lambda(w) \\
\text{Sing}(u)  & \ &\{x \in \text{FB}(u) \, | \, x \text{ is a singular point}\} \\
\text{Reg}(u)  & \ &\{x \in \text{FB}(u) \, | \, x \text{ is a regular point}\} \\
B_{r}(p) & \ & \{ x \in \mathcal{M} : \dist_p(x) < r \} \\
\eta_{\delta}(S) & \ & \text{the} \ \delta\text{-neighborhood of the set} \ S \\
D_r(p) & \ & \text{the Mean Value set for the point} \ p \ \text{with ``radius''} \ r \\
\Delta_g & \ & \text{the Laplace-Beltrami operator on} \ \mathcal{M}. \\
\end{array}
$$
In particular since $D_r(x_0) = \Omega(w_r)$ and since we will be using $D_e w_r$ to denote the derivative
of $w_r$ in the direction $e,$ it will often be less confusing to refer to $\Omega(w_r)$ instead of $D_r(x_0).$

Throughout the paper we assume that $a^{ij}(x)$ are bounded, symmetric, and uniformly elliptic,
and we define the divergence form elliptic operator
\begin{equation}
     L := D_j \, a^{ij}(x) D_i \;,
\label{eq:Ldef}
\end{equation}
or, in other words, for a function $u \in W^{1,2}(\Omega)$ and $f \in L^2(\Omega)$ we say ``$Lu = f$ in $\Omega$'' if
for any $\phi \in W_{0}^{1,2}(\Omega)$ we have:
\begin{equation}
    - \int_{\Omega} a^{ij}(x) D_{i} u D_{j} \phi = \int_{\Omega} f \phi \;.
\label{eq:Ldef2}
\end{equation}
(Notice that with our sign conventions we can have $L = \Delta$ but not $L = -\Delta.$)
With our operator $L$ we let $G(x,y)$ denote the Green's function for all of $\R^n$ and observe that the
existence of $G$ at least on bounded sets and in $\R^n$ when $n > 3$ is guaranteed by the work of Littman,
Stampacchia, and Weinberger.  (See \cite{LSW}.)  For definitions of Sobolev spaces and for standard theorems
about uniformly elliptic divergence form operators on $\R^n$ we refer the reader to the excellent text by
Gilbarg and Trudinger and for the situation on a Riemannian manifold we mainly use volume one of the sequence
of PDE texts by Michael Taylor.  (See \cite{GT} and \cite{MET}.)

Finally, there is a simple lemma in \cite{B} that we will use repeatedly, so we record it here for the reader's convenience:
\begin{lemma}[Theorem 2.7c of \cite{B}]  \label{B27c}
Suppose that for $i = 1,2,$ the functions $w_i \geq 0$ solve the obstacle problem:
    \BVPnA{\Delta u}{\bigrchi_{ \{ w > 0 \} } g}{u}{\psi_i}{KeyObProb27c}{B_1}
where $0 < \mu_1 \leq g \leq \mu_2,$ and $\psi_1 \leq \psi_2 \leq \psi_1 + \epsilon,$ then
$$w_1 \leq w_2 \leq w_1 + \epsilon$$
and in particular
$$||w_1 - w_2||_{L^{\infty}(B_1)} \leq \epsilon \;.$$
\end{lemma}

\section{Proof of Compact Containment of Mean Value Sets, Part I}  \label{FirstApp}

We assume that $L := \partial_i (a^{ij}(x) \partial_j),$ that $||a^{ij}||_{C^{1,1}} < \infty,$ and as above
we let $u_r$ denote the solution to
\BVPnA{L(u)}{- \bigrchi_{\{u < G\}} r^{-n}}{u}{G(\cdot,x_0)}{KeyObProb2}{B_M(x_0)}
and let $w_r(x) := G(x, x_0) - u_r(x).$

\begin{lemma} \label{Lemma:L(Dewr)}
$L(D_e w_r)$ is a function such that,
$$|L(D_e w_r)| \leq C(\rho) < \infty \text{ in } \Omega(w_r) \setminus B_{\rho}(x_0)$$
for any direction $e$ and $\rho>0$ so that $\qhclosure{B_{\rho}(x_0)} \; \subset \Omega(w_r).$
\end{lemma}
\begin{proof}
Define $E:=\Omega(w_r) \setminus B_{\rho}(x_0)$ and let $\phi \in C^{\infty}_0(E).$
\begin{alignat*}{1}
     -\int_{E} a^{ij}D_j (D_e w_r) D_i \phi \, dx
                &=  \int_E D_e(a^{ij} D_i \phi) D_j w_r \, dx \\
                &= \int_E D_i \phi (D_e a^{ij}) D_j w_r \, dx + \int_E (D_e D_i \phi) a^{ij} D_j w_r \, dx.
\end{alignat*}
On the other hand, since $D_e \phi \in C^{\infty}_0(E)$ is a permissible test function, the
second integral turns out to be zero:
\begin{alignat*}{1}
         \int_E (D_e D_i \phi) a^{ij} D_j w_r \, dx
                &= \int_E D_i (D_e \phi) a^{ij} D_j w_r \, dx \\
                &= - \int_E D_e \phi \, r^{-n} \, dx \\
                &= \int_E \phi \, D_e  r^{-n} \, dx \\
                &= 0 \;.
\end{alignat*}
Hence, we have $L(D_e w_r) = - D_i (D_e \, a^{ij} D_j w_r) \in L^{\infty}(E),$ with a uniform bound
since we have excised a ball around the singularity.
\end{proof}

\noindent
\textit{Proof of Theorem\refthm{CCMVT1}\!\!.}
From \cite{BH1} we know that $\Omega(w_r) \subset \Omega(w_s).$ Hence, we need only show that there does not
exist a point $q \in FB(w_r) \cap FB(w_s).$ In order to show this we will consider the function $v:=w_s - w_r$
which satisfies:
\begin{enumerate}
     \item $v \geq 0$ in $B_M$
     \item $v=w_s \geq 0$ on $\partial \Omega(w_r)$
     \item $Lv=s^{-n} - r^{-n}<0$ in $\Omega(w_r)$
     \item $v>0$ in $\Omega(w_r)$
\end{enumerate}

Assume that there exists a point $q \in FB(w_r) \cap FB(w_s).$ Consider the function $D_e w_r$ for some unit
vector $e$ to be chosen later. Lemma \ref{Lemma:L(Dewr)} ensures that in the set
$\Omega(w_r) \setminus B_{\rho}$, for small $\rho$, there exists $\epsilon_1>0$ such that
$$L(v - \epsilon_1 D_e w_r)<0 \text{ in } \Omega(w_r) \setminus B_{\rho}.$$
Also, note that, for $\rho$ small enough, $v > 0$ on $\partial B_{\rho}$ by \cite[Lemma 6.2]{BH1}, nondegeneracy,
and optimal regularity. Hence, there exists $\epsilon_2>0$ such that
$$v - \epsilon_2 D_e w_r \geq 0 \text{ on } \partial B_{\rho} \,.$$
Now, in fact, for any $\epsilon_2$ whatsoever, by standard regularity results for the obstacle problem
(see \cite{C1}, \cite{C2}, or \cite{B}) we automatically have
$$v - \epsilon_2 D_e w_r = v = w_s \geq 0 \text{ on } \partial \Omega(w_r) \,.$$
Then, by the Weak Maximum Principle 
$$v - \epsilon D_e w_r \geq 0 \text{ in } \Omega(w_r) \setminus B_{\rho}$$
for $\epsilon=min\{\epsilon_1,\epsilon_2\}.$ However, by optimal regularity and nondegeneracy we know that
$$\sup_{B_{\delta}(q)}v \leq C_1 \delta^2 \text { and } \sup_{B_{\delta}(q)}|\nabla w_r| \geq C_2 \delta$$
for $\delta>0$ such that $\overline{B_{\delta}(q)} \subset B_M.$ Therefore, for $\delta$ small enough, there
exists a point $y \in B_{\delta}(q)$ and a unit vector $e$ so that
$$v(y) - \epsilon D_e w_r(y) \leq C_1 \delta^2 - \epsilon C_2 \delta<0$$
which gives us a contradiction.
\qed

\section{Proof of Compact Containment of Mean Value Sets, Part II}  \label{SecondApp}

Now we turn to the proof of Theorem\refthm{CCMVT2}\!\!.  Before starting, however, it is worth
noting how the previous proof fails in this case.  Perhaps the greatest problem is the inability to
define a direction $e$ globally.  Accordingly, the set $D_r(x_0) \setminus B_{\rho}(x_0)$ which
could be huge (and therefore nowhere close to being contained within a chart of the manifold
$\mathcal{M}$) cannot be used for our argument.  We must work locally 
and so instead of working on $D_r(x_0) \setminus B_{\rho}(x_0),$ we work on
$D_r(x_0) \cap B_{\delta}(q)$ where $q \in \partial D_r(x_0) \cap \partial D_s(x_0).$  On this
new set, however, although we have no problem defining directions as long as $\delta$ is 
sufficiently small, we have a new problem of potentially having our test function being negative
on parts of the boundary.

The setting we have for this section assumes that we have a point
$q \in \partial D_r(x_0) \cap \partial D_s(x_0),$ and a $\delta > 0$ that is small enough so
that
\begin{enumerate}
      \item $B_{\delta}(q)$ is completely contained within a single chart $(\mathcal{U}, \varphi)$ of $\mathcal{M},$
      \item we let $y$ be points within the original manifold, and $x$ denote points in $\varphi(\mathcal{U})$ so that
$x = \varphi(y),$ and
      \item we assume that the $\varphi$ is giving us normal coordinates around $q$ and then 
the operator $\Delta_g$ can be expressed:
          \begin{equation}
              \begin{array}{rl}
              \displaystyle{\Delta_g u(y)} &= \rule[-4.5ex]{0pt}{7ex}
                      \displaystyle{\frac{1}{\sqrt{ |det \; g(x)| }} \cdot \frac{\partial}{\partial x_i} 
                                            \left( g^{ij}(x) \sqrt{ |det \; g(x)| } \; 
                                            \frac{\partial}{\partial x_j} u(x) \right)} \\
                                &=: \displaystyle{g^{ij}(x) \frac{\partial}{\partial x_i} \frac{\partial}{\partial x_j} u(x) 
                                  + b^{j}(x) \frac{\partial}{\partial x_j} u(x)} \\
                                &=: \displaystyle{Lu(x) \;,}
              \end{array}
          \label{eq:BigUgly}
          \end{equation}
with $g^{ij}(p) \rightarrow \delta^{ij},$ and $b^{i}(p) \rightarrow 0$ as $p \rightarrow q.$  (We are
using $\delta^{ij}$ to denote the Kronecker delta.)
\end{enumerate}
So the picture that we have on the manifold is given in Figure~\ref{fig:ManifoldPicture}.
\begin{figure}[!h]
	\centering
	\scalebox{.85}{\includegraphics{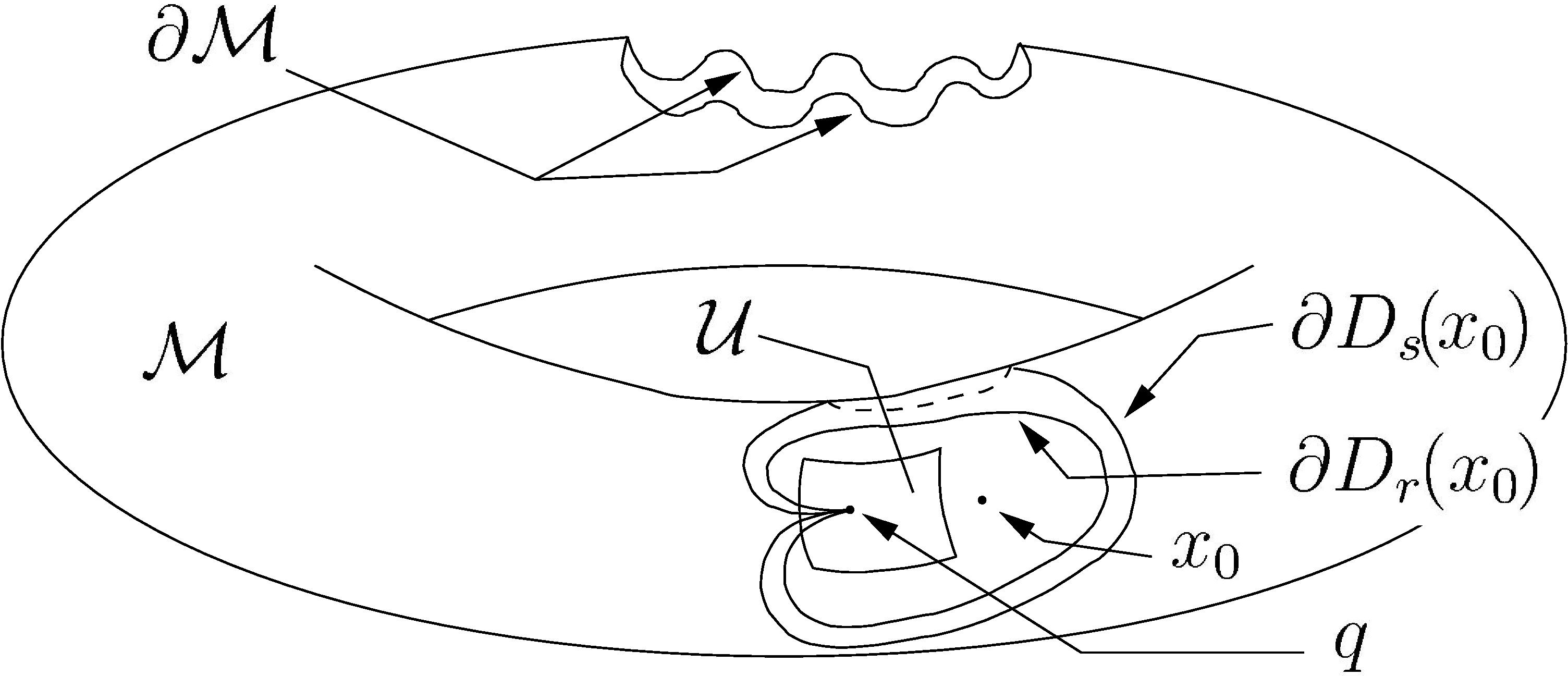}}
	\caption{The Picture on the Manifold $\mathcal{M}.$}
	\label{fig:ManifoldPicture}
\end{figure}
In terms of a source for the differential geometry facts and conventions that we needed and used, we
found the text \cite{Au} by Aubin to be useful for everything above.

\begin{remark}
An astute reader might complain that our mean value sets in Figure~\ref{fig:ManifoldPicture} lack the
reflection symmetry that would be enjoyed on a piece of a perfect torus, so our picture should be
considered to be a ``cartoon'' in this respect.
\end{remark}

Having seen the situation on the manifold above, we observe that in this section we can do all of our work
within the chart $\mathcal{U}$ and so we can view our entire problem in the local picture found in
$\mathcal{V} := \varphi(\mathcal{U}) \subset \R^n,$ and this fact allows us to get away with some obvious
abuses of notation.  Indeed, we will use $q, D_r(x_0),$ and $D_s(x_0)$ as shorthand for
$\varphi(q), \varphi(D_r(x_0) \cap \mathcal{U}),$ and $\varphi(D_s(x_0) \cap \mathcal{U})$ respectively.
Since it will be convenient to work with a perfect ball in $\mathcal{V},$ we use $B_{\epsilon}(q)$ to denote the
largest ball centered at $\varphi(q)$ which is contained in $\varphi(B_{\delta}(q)).$
So within $\mathcal{V}$ we have a nondivergence form elliptic operator, $L,$ which we
can take to be defined on $C^{2}(\mathcal{V} \cap D_r(x_0))$ and which converges to the Laplacian in the sense
described above as we zoom in on $q.$   Lastly, we will obviously view all of our solutions to obstacle
problems (so $u_r$ and $w_r$ for example) as being functions defined on $\mathcal{V}.$  All of these
conventions lead to the local picture shown in Figure~\ref{fig:LocalPicture}.
\begin{figure}[!h]
	\centering
	\scalebox{.85}{\includegraphics{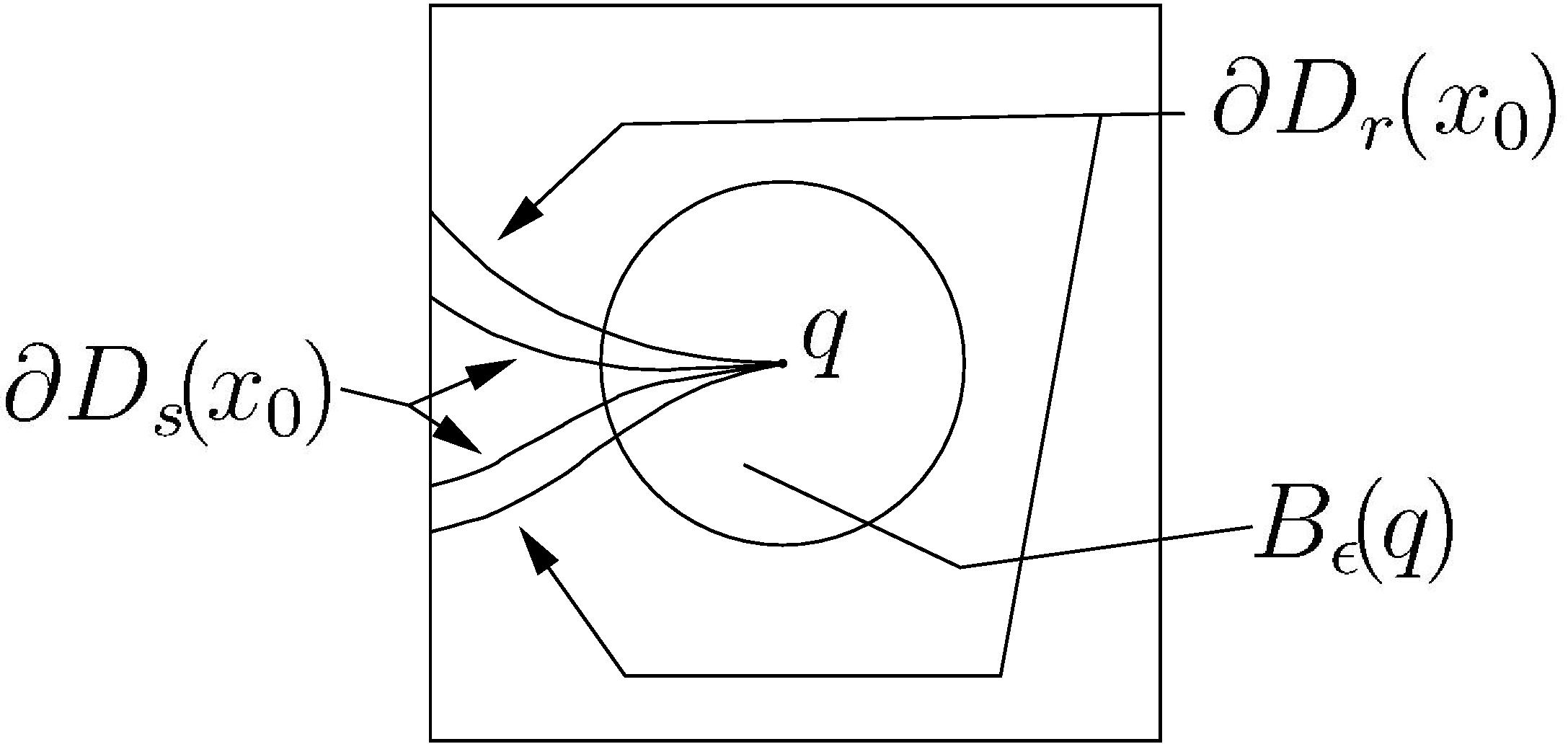}}
	\caption{The Local Picture in $\mathcal{V}.$}
	\label{fig:LocalPicture}
\end{figure}

Before jumping into the main proof, we observe the following two lemmas:
\begin{lemma}[Barrier Function Estimates]   \label{BFE}
By shrinking $\epsilon$ if necessary, we have
\begin{equation}
    2n - 1 \leq L( |x - y|^2 ) \leq 2n +1 \;,
\label{eq:Ldist2}
\end{equation}
for all $x \in B_{\epsilon}(q)$ and for any fixed $y \in B_{\epsilon}(q).$
\end{lemma}
\begin{proof}
This estimate follows immediately by using Equation\refeqn{BigUgly}along with the fact that
$g^{ij}(x) \rightarrow \delta^{ij}$ and $b^{i}(x) \rightarrow 0$ as $x \rightarrow q.$
\end{proof}

\begin{lemma}[Boundedness Estimate]   \label{BddEst}
For $x \in D_r(x_0) \cap \mathcal{V}$ and any direction $e$ we have
\begin{equation}
      |L(D_e w_r)| \leq C < \infty \;.
\label{eq:BddEstDeriv}
\end{equation}
\end{lemma}
\begin{proof}
We observe that this estimate is very similar to the estimate in the previous section given in
Lemma\refthm{Lemma:L(Dewr)}\!\!.  On the other hand, this time the proof is easier.  We know that
in $D_r(x_0) \cap \mathcal{V}$ we have
$$L w_r = g^{ij}D_{ij} w_r + b^{i} D_i w_r = r^{-n} \;.$$
Differentiating this equation in the $e$ direction, we have:
\begin{alignat*}{1}
   0 &= D_e \left( g^{ij}D_{ij} w_r + b^{i} D_i w_r \right) \\
      &= L(D_e w_r) + (D_e g^{ij}) D_{ij} w_r + (D_e b^i) D_i w_r \;,
\end{alignat*}
so by using regularity known for solutions of the obstacle problem along with the regularity that
we have for the coefficients in our operator $L,$ we conclude that
$$|L(D_e w_r)| \leq |D_e g^{ij}| \cdot |D_{ij} w_r| + |D_e b^i| \cdot |D_i w_r| \leq C < \infty.$$
\end{proof}

\noindent
\textit{Proof of Theorem\refthm{CCMVT2}\!\!.} We can assume by shrinking $\epsilon$ again if necessary,
that $x_0 \notin B_{\epsilon}(q)$ and $B_{\epsilon}(q)$ has no intersection with $\partial \mathcal{M}$ if
$\mathcal{M}$ has boundary.  (Lemma 6.2 of \cite{BH1} guarantees that we can find such an $\epsilon.$)
Now we consider the function
\begin{equation}
     h := w_s - w_r - \mu D_{e} w_r
\label{eq:hdef}
\end{equation}
where $\mu > 0$ will be a very small number and $e$ will be a direction to be chosen later.  We are going to
arrive at a contradiction by showing that $h \geq 0$ in a ball around $q$ intersected with $D_r(x_0)$ while
using the asymptotics of the functions which make up $h$ along with a good choice of the direction $e$ allow
us to show that $h$ must be negative arbitrarily close to $q$ within $D_r(x_0).$

Now for any positive $\rho < \epsilon,$ we consider the set
$E_{\rho} := B_{\rho}(q) \cap D_r(x_0) = B_{\rho}(q) \cap \Omega(w_r).$ Within this set we have
$w_s - w_r \geq 0$ and $L(w_s - w_r) = s^{-n} - r^{-n} < 0.$
Hence, in $E_{\rho} \setminus \eta_{\gamma}(\partial E_{\rho})$ there exists a $\kappa$ such that
$w_s - w_r \geq \kappa > 0.$  Having made this observation, it turns out that we will need a more precise
lower bound, and by using the estimate from Lemma\refthm{BFE}we will succeed.  Along these lines we first
shrink $\epsilon$ (and therefore $\rho$) if necessary to be sure that that estimate applies, and we assume
that $z \in E_{\rho} \setminus \eta_{\gamma}(\partial E_{\rho})$ and observe that this implies that
$B_{\gamma}(z) \subset E_{\rho}.$  Next we define
$$\Theta(x) := \frac{(s^{-n} - r^{-n})(|x - z|^2 - \gamma^2)}{6n}$$
for use as a barrier function.  Indeed, observe that
\begin{enumerate}
    \item $\Theta = 0 \leq w_s - w_r$ on $\partial B_{\gamma}(z),$ and
    \item recalling that $L(w_s - w_r) = s^{-n} - r^{-n} < 0$ and using the last lemma we get:
\begin{alignat*}{1}
      L\Theta &= \frac{s^{-n} - r^{-n}}{6n} \cdot L( |x - z|^2 ) \\
                   &\geq (s^{-n} - r^{-n}) \cdot \left( \frac{2n + 1}{6n} \right) \\
                   &\geq L(w_s - w_r) \ \ \ \text{in} \ \ \ B_{\gamma}(z).
\end{alignat*}
\end{enumerate}
Thus, by using the weak maximum principle we have
$$w_s(z) - w_r(z) \geq \frac{\gamma^2 (r^{-n} - s^{-n})}{6n} \ \ \text{for all} \ z \ \text{within} \ \ 
    E_{\rho} \setminus \eta_{\gamma}(\partial E_{\rho})\;.$$

We can now observe the following properties of $h:$
\begin{enumerate}
    \item By assuming that $\mu$ is sufficiently small, we have
\begin{equation}
Lh = s^{-n} - r^{-n} - \mu L(D_e w_r) \leq - \alpha < 0  \ \text{in} \  E_{\epsilon} \;.
\label{eq:Lhestim}
\end{equation}
    \item $h = w_s \geq 0$ on $\partial \Omega(w_r).$
    \item Within $E_{\rho} \setminus \eta_{\gamma}(\partial \Omega(w_r)),$ by assuming that $\mu$ is sufficiently small, we have
\begin{equation}
    h \geq \frac{\gamma^2 (r^{-n} - s^{-n})}{6n} - \mu D_e w_r \geq \frac{\gamma^2 (r^{-n} - s^{-n})}{10n} \;.
\label{eq:AwayBO}
\end{equation}
    \item Within $E_{\rho} \cap \eta_{\gamma}(\partial \Omega(w_r)),$ by using the optimal gradient bounds for $w_r,$ we have
\begin{equation}
     h \geq - \mu \gamma C \;.
\label{eq:NearBO}
\end{equation}
\end{enumerate}
The picture can be seen in Figure \ref{fig:ForcePositive}.
\begin{figure}[!h]
	\centering
	\scalebox{.65}{\includegraphics{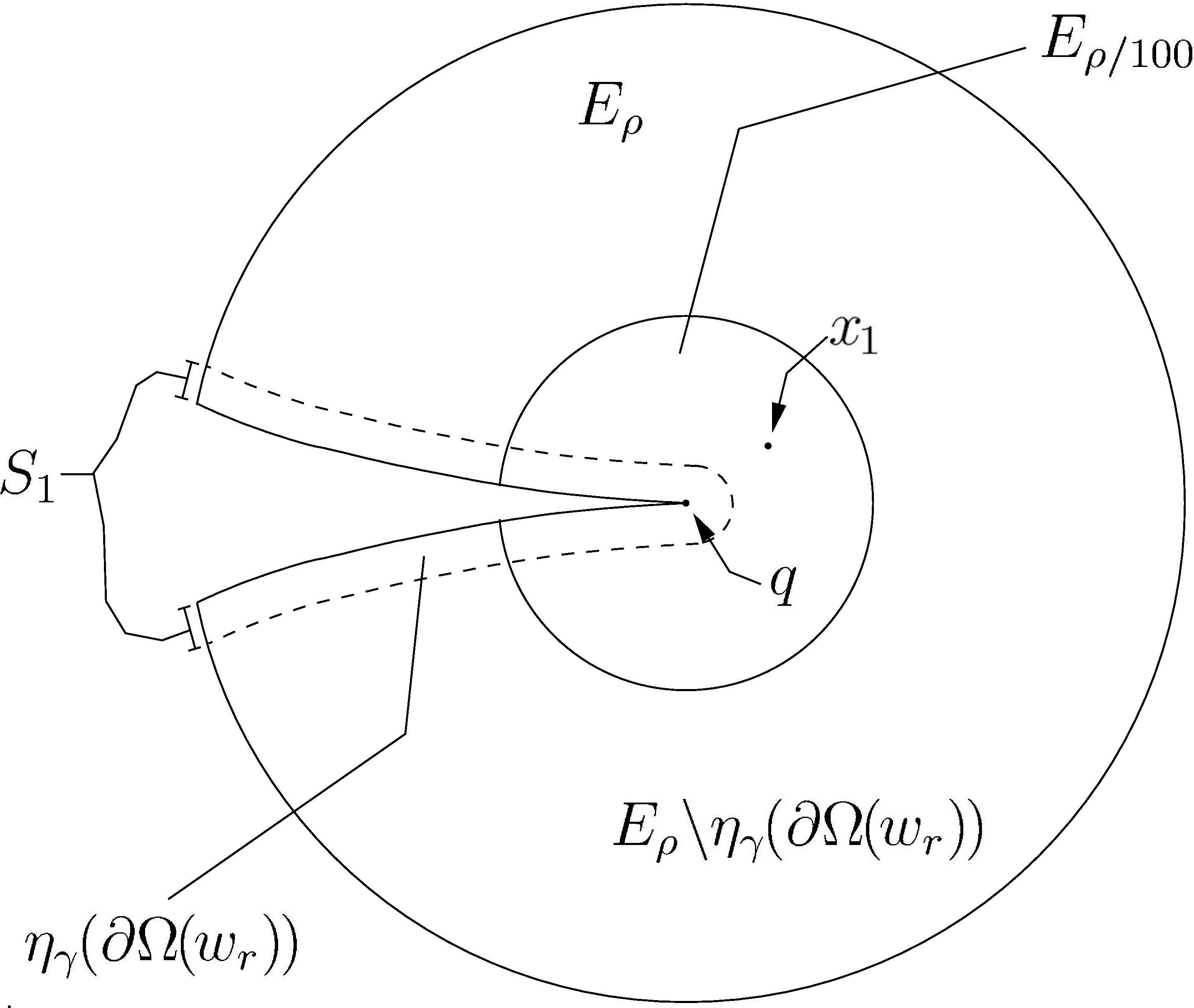}}
	\caption{The Picture in $B_{\rho}(q)$}
	\label{fig:ForcePositive}
\end{figure}

We are now in position to use the ideas within \cite[Lemma 11]{C2} in order to show that $h$ must be nonnegative
everywhere in a small enough ball around $q.$  On the other hand, for the sake of keeping this article more
self-contained, and because of slight changes that need to be made (largely because we have an operator
which is \textit{close} to the Laplacian, and not exactly the Laplacian) we will present the argument here.  In any
case we claim that $h \geq 0$ within $E_{\rho/100} = B_{\rho/100}(q) \cap \Omega(w_r)$ provided $\mu$ is sufficiently
small.

To begin the proof of our claim, we assume that there exists an $x_1 \in E_{\rho/100}$ with $h(x_1) < 0.$  We now
define
\begin{equation}
   v(x) := h(x) + \delta \left( \frac{r^{-n}}{4n}|x - x_1|^2 - w_r \right) \;,
\label{eq:UglyVdef}
\end{equation}
and we observe that $v(x_1) < 0.$  We also know that 
$$Lv \leq -\alpha + \delta\left( \frac{r^{-n}}{4n}(2n +1) - r^{-n} \right) \leq -\alpha < 0$$
in all of $E_{\epsilon}$ by using Lemma\refthm{BFE}and Equation\refeqn{Lhestim}\!.  So, by applying the weak
maximum principle, we can be sure that $v$ must attain a negative minimum on $\partial E_{\rho}.$  On the other hand,
all along $\partial \Omega(w_r),$ by using the definition of $h(x)$ we have $v(x) = w_s(x) + C\delta|x - x_1|^2 > 0.$
So, we know that $v(x)$ attains its negative mimimum on $\partial B_{\rho}(q) \cap \partial E_{\rho}.$  For this
remaining piece of the boundary, it is convenient to split it into
$S_1 :=  \eta_{\gamma}(\partial \Omega(w_r)) \cap \partial B_{\rho}(q)$ and 
$S_2 := \partial B_{\rho}(q) \setminus \eta_{\gamma}(\partial \Omega(w_r)),$ and then by
employing Equations\refeqn{NearBO}and\refeqn{AwayBO}on those sets respectively, we get:
\begin{equation}
  v \geq  - C_1 \mu \gamma + C_2 \delta r^{-n} \rho^2 - C_3 \delta \gamma^2 \ \ \ \ \text{on} \ 
  S_1
\label{eq:vNBO}
\end{equation}
and
\begin{equation}
  v \geq C_4 (r^{-n} - s^{-n}) \gamma^2 + \delta (C_5 r^{-n} \rho^2 - w_r) \ \ \ \ \text{on} \ 
  S_2 \;.
\label{eq:vABO}
\end{equation}
On both sets we wish to choose constants so that $v$ is forced to be nonnegative.  For $S_1$
we choose $\gamma <\!\!< \rho$ to force $C_2 r^{-n} \rho^2 > C_3 \gamma^2$ and then choose
$\mu$ as small as we need to give us the desired inequality.  For $S_2$ we choose $\delta <\!\!< \gamma^2$
and then shrink $\mu$ again if needed to fix the inequality on $S_1.$  So, at this point we have a contradiction
to any negativity of $v$ within $E_{\rho/100}.$

Now, just as in the end of the proof of Theorem\refthm{CCMVT1}\!, it follows from standard 
regularity and nondegeneracy estimates for the obstacle problem, that $w_s$ and $w_r$ are bounded by
a constant times $|x - q|^2$ within $E_{\rho/100},$ while $D_e w_r(x)$ must grow linearly for some
choice of $e$ within the same set.  Now by replacing $e$ with $-e$ if necessary, we get $h < 0$ somewhere
within $E_{\rho/100}$ and we have the desired contradiction.
\qed

\begin{remark}[Existence of singular points in mean value sets]   \label{ESPMVS}
Currently, it is unknown whether or not mean value sets of the type described in the previous section
\textit{ever} possess singular points.  So, there is an outside chance that they do not exist.  Having made
this observation, it is a rather simple matter to show the existence of mean value sets for the Laplace-Beltrami
operator on manifolds which have singular points.  Indeed, at the moment the topology of one of these
sets changes, you will necessarily have singular points.  (We can also say that by using the results within
\cite{AB} the free boundary won't ``jump'' from a configuration with one topology where the set is smooth
to a different topology with smooth boundary;  there will always be a moment with a ``collision.'')  For a
concrete example, consider harmonic functions on a typical cylinder.  Obviously for any such function, one
can ``unroll'' the cylinder and get a periodic harmonic function on $\R^2.$  Kuran proved that any connected
mean value set for the point $x_0 \in \R^n$ which has positive measure and which contains $x_0$ must be
(up to a set of measure zero) a ball centered at $x_0$ (\cite{Ku}).  So, the $D_r(x_0)$ which fit within one
period should be disks centered at $x_0.$  By increasing the radius of the disk until the diameter is the length
of a period, we get a mean value set which when viewed on the original cylinder, will have a ``double cusp.''
Thus, we can be certain that the proof that we gave of the theorem in this section doesn't apply only to the
empty set.
\end{remark}

\section{Singular Point Approximation}  \label{SingPtApprox}

As before in Section \ref{FirstApp} we consider an operator $L := \partial_i (a^{ij}(x) \partial_j),$ but
now, although we are still working with the obstacle problem, we are no longer working with mean-value
sets, and currently we will only assume a bound on $||a^{ij}||_{C^{0}(\closure{B_1})}.$  Because our
coefficients are always continuous, we can assume without loss of generality that $a^{ij}(0) = \delta^{ij}$
by changing coordinates.  In this setting, we let $w \in W^{1,2}(B_1)$ satisfy:
\begin{equation} \label{eq:wsat}
\begin{cases}
   Lu =\bigrchi_{\{u>0\}}  \quad \text{ in } B_1 \\
   u \geq 0 \\
   0 \in \text{Sing}(u).
\end{cases}
\end{equation}
Next, for any $r < 1,$ and any $t \in \R,$
we let $u_{r;t} \in W^{1,2}(B_r)$ to be the solution to
\begin{equation} \label{eq:urtsolves}
\begin{cases}
\Delta u = \bigrchi_{\{u>0\}} \quad &\text{ in } B_r\\
u=(w+t)^+ \quad &\text{ on } \partial B_r \\
u \geq 0
\end{cases}
\end{equation}
with the goal in this section of getting $u_{r,t}$ to approximate $w$ and to also have a singular free boundary point
at $0.$  One reason why we had this goal, was because we had hoped to generalize the regularity results of Figalli
and Serra (\cite{FS}) to obstacle problems with more general elliptic operators than simply the Laplacian.
Toward this aim, we will work with quadratic rescalings of $w$ and $u_{r;t},$ and with $T := t/r^2,$
we make the following list of definitions:
\begin{alignat*}{1}
     w_r(x)         &:= \frac{w(rx)}{r^2} \\
     v_{r;T}(x)   &:= \frac{u_{r;t}(rx)}{r^2} \\
     a^{ij}_r(x)  &:= a^{ij}(rx) \\
     L_r               &:= D_i(a^{ij}_r(x)D_j).
\end{alignat*}
Then we observe that $w_r$ satisfies
\begin{equation} \label{eq:wrprob}
\begin{cases}
    L_ru =\bigrchi_{\{u>0\}}  \quad \text{ in } B_1 \\
    u \geq 0 \\
    0 \in \text{Sing}(u), \\
\end{cases}
\end{equation}
and $v_{r;T}$ is the solution to
\begin{equation} \label{eq:vrTsolves}
\begin{cases}
\Delta u = \bigrchi_{\{u>0\}} \quad &\text{ in } B_1\\
u=(w_r+T)^+ \quad &\text{ on } \partial B_1 \\
u \geq 0.
\end{cases}
\end{equation}
Furthermore, we observe that by decreasing $r$ we can make
$$||a^{ij}_{r}(\cdot) - \delta^{ij}||_{C^{0}(\closure{B_1})}$$
as small as we like.

\begin{lemma}[Getting $0$ into FB]  \label{GZFB}
Given $0 < r \leq 1,$ and defining $w_r$ and $v_{r;T}$ as above, there exists an $S = S(r)$ so that
$0 \in FB(v_{r;S}).$
\end{lemma}

\begin{proof}
We define the set $I := \left\{ T \in \R \ \rule[-.03in]{.01in}{.18in} \ 0 \in \text{int}(\Lambda(v_{r;T})) \right\}\!,$
and observe that $I$ is a bounded nonempty set.  Indeed, if $T$ is sufficiently negative, then $v_{r;T} \equiv 0,$
and if $T$ is more than $1/2n,$ then $v_{r;T}(x) > |x|^2/2n.$  Indeed, the following inclusions follow from those
observations:
\begin{equation}
      ( - \infty, - \max_{B_1}{w_r}) \subset I \subset ( - \infty, 1/2n )
\label{eq:SetBounds}
\end{equation}
So, we let $S := \sup I,$ and we claim that $0 \in FB(v_{r;S}).$

Suppose not.  Then either $0 \in \Omega(v_{r;S})$ or $0 \in \text{int}(\Lambda(v_{r;S})).$  In the first case
we have a closed ball $\qclosure{B_{\epsilon}(0)} \, \subset \Omega(v_{r;S}),$ and so in this case we let
$$\gamma := \min \left\{ v_{r;S}(x) \ \rule[-.03in]{.01in}{.18in} \ x \in \; \qclosure{B_{\epsilon}(0)} \right\} \;.$$
Now we define $\tilde{S} := S - \gamma/2$ and by using Lemma\refthm{B27c}we have
$$\gamma \leq v_{r;S}(0) \leq v_{r;\tilde{S}}(0) + \frac{\gamma}{2} \;,$$
and this inequality implies $\sup I \leq \tilde{S} < S$ which is a contradiction.

In the other case, we have a closed ball $\qclosure{B_{\epsilon}(0)} \, \subset \Lambda(v_{r;S}).$  Now we let
$T_j \downarrow S$ and observe that by the definition of $S,$ we
have $0 \in \, \pclosure{\, \Omega(v_{r;T_j})}$ for all $T_j.$  Using the standard nondegeneracy results for the
obstacle problem, for every $j,$ we have an $x_j$ in $B_{\epsilon/2}(0)$ with
$$v_{r;T_j}(x_j) \geq C \epsilon^2 \;.$$
This inequality leads to,
$$C \epsilon^2 \leq ||v_{r;S}- v_{r;T_j}||_{L^{\infty}(B_{\epsilon}
(0))} \leq ||v_{r;S}- v_{r;T_j}||_{L^{\infty}(B_1(0))} \leq ||v_{r;S}- 
v_{r;T_j}||_{L^{\infty}(\partial B_1(0))},$$
where the last inequality is by Lemma\refthm{B27c}again. However, since
$$ ||v_{r;S}- v_{r;T_j}||_{L^{\infty}(\partial B_1(0))} \leq ||S-T_j ||
_{L^{\infty}(\partial B_1(0))} \rightarrow 0 \;,$$
we have a contradiction.
\end{proof}

\begin{lemma}[Uniqueness of $S$]  \label{UoS}
The $S(r)$ given in Lemma\refthm{GZFB}is the only number $S,$ such that $0 \in FB(v_{r;S}).$
\end{lemma}

\begin{proof} Suppose not.  Then there exists $S_1 < S_2$ such that $0 \in FB(v_{r;S_1}) \cap FB(v_{r;S_2}).$
Now from Lemma\refthm{B27c}we know that $\Omega(v_{r;S_1}) \subset \Omega(v_{r;S_2}),$ and we
consider the function
\begin{equation}
    V := v_{r;S_2} - v_{r;S_1} \,.
\label{eq:BigVdef}
\end{equation}
We observe that $V = S_2 - S_1 > 0$ on $\partial B_1 \cap \Omega(v_{r;S_1}),$ and $V \in C^{0}(\closure{B_1}).$
By using the continuity of $V,$ we have a $\delta \in (3/4, 1)$ so that $V \geq \frac{1}{2}(S_2 - S_1)$ on
$\partial B_{\delta} \cap \Omega(v_{r;S_1}).$  We now define the function
\begin{equation}
    h:= V - \mu D_e v_{r;S_1}
\label{eq:newhdef}
\end{equation}
for a direction $e$ to be chosen later.  We observe that $h$ is harmonic, and because
$v_{r;S_1} \in C^{1,1}(\dclosure{B_{\delta}})$ we can choose $\mu$ to be sufficiently small so that
$h > 0$ on $\partial B_{\delta} \cap \Omega(v_{r;S_1}).$  Now by the optimal regularity results for the
obstacle problem $v_{r;S_1} = D_e v_{r;S_1} = 0$
on all of $\partial \Omega(v_{r;S_1}) \cap B_1,$ so we observe that $h \geq 0$ on all of
$\partial (B_{\delta} \cap \Omega(v_{r;S_1})).$  Thus, the maximum principle gives us $h \geq 0$ in all
of $B_{\delta} \cap \Omega(v_{r;S_1}).$  Now by proceeding exactly as in the proof of
Theorem\refthm{CCMVT1}where we use the asymptotics of the functions making up $h$ to find a spot
where it is negative (and assigning an appropriate direction $e$) we get a contradiction.
\end{proof}

\begin{lemma}[$S(r) \rightarrow 0$ as $r \rightarrow 0$]   \label{SgtZ}
For the $S(r)$ given in Lemma\refthm{GZFB}we have
\begin{equation}
     \lim_{r \rightarrow 0} S(r) = 0 \;.
\label{eq:SgtZ}
\end{equation}
\end{lemma}

\begin{proof} Suppose not.  Then since the $S(r)$ are uniformly bounded, we can find a sequence
$r_j \rightarrow 0$ such that $S(r_j) \rightarrow \tilde{S} \ne 0.$  By our assumptions about $w_r,$
and by applying \cite[Lemma 3.1 and Lemma 3.2]{BH2}
we know that $w_{r_j} \rightarrow w_0 = \frac{1}{2}(x^{T} Mx)$ uniformly where $M$ is a
nonnegative matrix and $\Delta w_0 = \text{Trace}(M) = 1.$  So, we know that $w_0$ satisfies
\begin{equation} \label{eq:w0sat}
\begin{cases}
   \Delta u =\bigrchi_{\{u>0\}}  \quad \ &\text{ in } B_1 \\
   u(x) = x^{T} Mx \quad &\text{ on } \partial B_1 \\
   u \geq 0
\end{cases}
\end{equation}
and additionally, $0 \in \text{Sing}(w_0).$  On the other hand, since $w_{r_j} \rightarrow \frac{1}{2}(x^{T} Mx)$
uniformly, we know that the boundary data of $v_{r_j;S_j}$ converges uniformly to
$(\frac{1}{2}(x^{T} Mx) + \tilde{S})^{+}.$  So, we have that the limit of the $v_{r_j;S_j},$ which we will call
``$v_{0;\tilde{S}}$'' satisfies
\begin{equation} \label{eq:v0sat}
\begin{cases}
   \Delta u =\bigrchi_{\{u>0\}}  \quad \ &\text{ in } B_1 \\
   u(x) = (x^{T} Mx + \tilde{S})^{+} \quad &\text{ on } \partial B_1 \\
   u \geq 0
\end{cases}
\end{equation}
and furthermore $0 \in \text{FB}(v_{0;\tilde{S}}).$  Now, since $\tilde{S} \ne 0$ we can use the functions
$w_0$ and $v_{0;\tilde{S}}$ along with Lemma\refthm{UoS}to get a contradiction.
\end{proof}

\begin{remark}[$v_{0;\tilde{S}} = w_0$]   \label{UsefulShit}
It follows from knowing that $\tilde{S} = 0$ along with Equations\refeqn{w0sat}and\refeqn{v0sat}and
the uniqueness of the solutions to such problems that $v_{0;\tilde{S}} = w_0$ and in particular
$0 \in \text{Sing}(v_{0;\tilde{S}}).$  We will use this fact in the next proof.
\end{remark}

Finally, to strengthen the statement of our final theorem, we follow \cite{FS} and for
$m \in \{ 0, 1, 2, \ldots, n - 1 \}$ we define the $m$-th stratum of the singular set to be the subset of
the singular set where the dimension of the kernel of the blow up limit is $m.$

\begin{theorem}[Preserving the Singular Point and Bounding the Stratum]   \label{PSP}
Given the function $w,$ there exists an $R > 0$ such that $0 \in \textnormal{Sing}(v_{r;S(r)})$ for all $r < R.$
Furthermore, by shrinking $R$ if necessary, this singular point is in the same or lower stratum as it is with $w.$
(i.e. If $0$ is in the $k$-th stratum of $w,$ then it will always belong to the strata for $v_{r;S(r)}$ with
$m \leq k$ for $r < R.$)
\end{theorem}

\begin{proof} Suppose there does not exist an $R > 0$ such that $0 \in \text{Sing}(v_{r;S(r)})$ for all $r < R.$
Then there exists $r_j \downarrow 0$ such that $0 \in \text{Reg}(v_{r_j;S(r_j)})$ for all $j.$  
Fix $\epsilon > 0$ to be chosen later, and
to simplify notation, we will let $v_j := v_{r_j;S(r_j)}$ for the duration of this proof.  Using our assumption, for each
$r_j$ there exists $n_j \in \partial B_1$ so that with $P_{r_j} := \max \{(x \cdot n_j),0\}^2$ we have
$$\frac{v_{r_j;S}(\rho x)}{\rho^2} \rightarrow P_{r_j}(x) \, \text{ uniformly as } \rho \rightarrow 0.$$
However, there exists a subsequence of $r_j$, which we denote again by 
$r_j,$ such that $n_j \rightarrow n_0$ and so
$$P_{r_j} \rightarrow P_0 = \max\{(x \cdot n_0),0\}^2$$
uniformly.  Hence, given any $\epsilon > 0$ there exists an $\mathcal{R}_1 > 0$ and a $\mathcal{K} > 0$ such 
that if $r _j < \mathcal{R}_1$ and $\rho < \mathcal{K},$ then
\begin{equation}
    \left| \left| \frac{v_{j}(\rho x)}{\rho^2} - P_0(x) \right| \right|_{\infty} \leq \epsilon.
\label{eq:Pep}
\end{equation}
On the other hand $v_{j} \rightarrow w_0 = \frac{1}{2}(x^TMx)$ uniformly in $\ohclosure{B_1}$ by
Remark\refthm{UsefulShit}\!, so there exists an $\mathcal{R}_2$ such that $r_j \leq \mathcal{R}_2$
implies
\begin{equation}
   \left| \left| \frac{v_{j}(\rho x)}{\rho^2} - x^{T} Mx \right| \right|_{\infty} \leq \epsilon.
\label{eq:Mep}
\end{equation}
By applying \cite[Lemma 2]{W} we know that there exists a constant $\gamma > 0$ so that
\begin{equation}
     \left| \left| x^{T} Mx - P_0(x) \right| \right|_{\infty} \geq \gamma \;.
\label{eq:WeisIsol}
\end{equation}
Now we fix $\epsilon < \gamma/2$ and use the triangle inequality combined with
Equations\refeqn{Pep}\!\!,\refeqn{Mep}\!\!, and\refeqn{WeisIsol}to get a contradiction.

At this point we have the first part of the theorem.  To show the second part we essentially repeat the
argument but replace the use of \cite[Lemma 2]{W} with the observation that a nonnegative matrix $M$
can be approximated arbitrarily well with matrices with lower dimensional kernels, but it will stay isolated from all of
the matrices with higher dimensional kernels.  To be more specific, the main difference from the first part
of the proof is that in place of $P_{r_j}$ we would have a sequence $x^{T} M_j \, x$ where the kernels of
the $M_j$ have dimension greater than the kernel of the $M,$ and this leads to a contradiction.
\end{proof}



\bibliographystyle{plain}
\bibliography{BIB}
\end{document}